\documentclass{amsart}
\usepackage{amssymb}
\usepackage{pb-diagram}

\usepackage{graphics}
\usepackage[all]{xy}
\usepackage{tikz}
\usetikzlibrary{arrows,automata}

\usetikzlibrary{decorations.pathreplacing}

\newtheorem{theorem}{Theorem}[section]

\newtheorem{prop}[theorem]{Proposition}

\newtheorem{claim}[theorem]{Claim}
\newtheorem{fact}[theorem]{Fact}

\newtheorem{lemma}[theorem]{Lemma}

\theoremstyle{definition}
\newtheorem{defn}[theorem]{Definition}

\newtheorem{example}[theorem]{Example}

\theoremstyle{remark}

\newtheorem{remark}[theorem]{Remark}

\newcommand{\la}{\langle}
\newcommand{\ra}{\rangle}

\newcommand{\sub}{\subseteq}

\newcommand{\pf}{\proof}

\def\Star{\mathrm{Star}}
\def\st{\mathrm{st}}

\newcommand{\cal}[1]{\ensuremath{\mathcal{#1}}}

\newcommand{\res}{\ensuremath{\upharpoonright}}

\newcommand{\ve}{\ensuremath{\varepsilon}}

\newcommand{\lam}{\ensuremath{\lambda}}

\newcommand{\Q}{\mathbb{Q}}

\begin{document}

\author {Pantelis  E. Eleftheriou}

\address{Department of Mathematics and Statistics, University of Konstanz, Box 216, 78457 Konstanz, Germany}

\email{panteleimon.eleftheriou@uni-konstanz.de}

\title{Semi-linear stars are contractible}
\date{\today}

\thanks{Research supported by an Independent Research Grant from the German Research
Foundation (DFG) and a Zukunftskolleg Research Fellowship.\newline
}


\begin{abstract}
Let $\cal R$ be an ordered vector space over an ordered division ring.
We prove that every definable set $X$ is a finite union of relatively open definable subsets which are definably simply-connected, settling a conjecture from \cite{eep1}. The proof goes through the stronger statement that the star of a cell in a special linear decomposition of $X$ is definably simply-connected. In fact, if the star is bounded, then it is definably contractible.
\end{abstract}

 \maketitle


\section{Introduction}
This paper deals with the general problem of covering definable sets in an o-minimal structure \cal R with  topologically nice subsets. It was first proved by Wilkie \cite{wilkie} that if \cal R expands an ordered field, then every bounded open definable set $X$ is a finite union of open cells. By Andrews \cite{andrews} and Edmundo-Eleftheriou-Prelli \cite{eep2}, the statement also holds if \cal R expands an  ordered group, but not a field (and, in fact, without assuming $X$ is bounded). In \cite{eep1}, some strong consequences of the above covering statements were derived and applied to the study of locally definable manifolds in \cal R. For instance, the existence of universal locally definable covering maps and the invariance of  o-minimal fundamental groups were established. Moreover, it was explained there how a positive solution to the following conjecture is crucial in extending those consequences to a much wider context (such as that of locally definable spaces).\\

\noindent\textbf{Conjecture (\cite[Section 5]{eep1}).} \emph{Every definable set $X$ is a finite union of relatively open definable subsets which are definably simply-connected.}\\

\noindent The virtue of the above statement over the  known aforementioned results is that $X$ is not assumed to be open. Later  in the introduction, we indicate why this weakening increases significantly the complexity of its proof. 

In this paper, we prove the conjecture in the semi-linear setting; that is, when \cal R is a pure ordered vector space.   As pointed out in \cite[Section 5]{eep1}, if \cal R expands an ordered field,  then the conjecture can be replaced by
(and, in fact, follows from)
the known triangulation theorem. The triangulation theorem fails in the semi-linear setting (see, for example, \cite[Example 1.1]{el-bsl}), creating a number of new difficulties and intricacies that do not exist in the field case. Indeed, by now semi-linear geometry has developed into a separate subject admitting a totally different set of techniques (see  \cite{el-affine} and \cite{lp}). In this paper, we advance semi-linear homotopy theory, featuring the notion of \emph{canonical retractions} of linear cells, which we  use to settle the conjecture in our setting. As a note, the  conjecture remains open in the `intermediate', semi-bounded case, where \cal R expands an ordered group  by a field whose universe is a bounded interval (\cite{ed-str}).

For the rest of this paper, we fix an ordered vector space  $\cal R=\la R, <, +, 0, \{x\mapsto \lam x\}_{\lam\in \Lambda}\ra$    over an ordered division ring $\Lambda$.
By `definable'  or `semi-linear' set, we mean a set definable in \cal R, possibly with parameters.  
The first-order theory of $\cal R$ is  well-understood (see, for example, \cite[Chapter 1, \S 7]{vdd-book} and \cite{elst}). In particular, $\cal R$ is o-minimal, and one can define the notion of a \emph{linear cell} and prove a linear cell decomposition theorem for all definable sets and functions (the precise terminology is postponed until Section \ref{preliminaries}).
It is worth noting that if $\cal R$ is the real vector space, then semi-linear geometry amounts to classical PL-topology. Our setting, however, includes far more general structures, whose role becomes increasingly prominent in areas such as that of valued fields and  non-archimedean tame topology (Hrushovski-Loeser \cite{hl}). Indeed, given an algebraically closed valued field $K$, its value group $\Gamma$ is known to be an ordered vector space over $\Q$, and  in \cite{hl}  strong connections between algebraic varieties over $K$ and the topology of $\Gamma$ were established. For example, the above authors proved in \cite[Theorem 13.1.3]{hl} an equivalence of categories between a certain homotopy category of definable subsets of quasi-projective varieties over $K$ and a  homotopy category of definable spaces over $\Gamma$.




We now proceed to describe the content and subtleties of this paper. The following definition is  essentially \cite[Definition 2.2]{eep2}.\\

\noindent\textbf{Definition.} (Stars) Let $X$ be a definable set,  $\mathcal{C}$ a linear  decomposition of $X$ and $C\in \cal C$.
The \emph{star of $C$ in $X$ with respect to $\mathcal{C}$}, denoted by $\st _{\mathcal{C}}(C)$, is the set
$$\st _{\mathcal{C}}(C)=\bigcup \{D\in \mathcal{C}: C \cap cl(D)\ne\emptyset\}.$$

\noindent The main difference from \cite[Definition 2.2]{eep2} is that here $\cal C$ is a linear decomposition of any definable set $X\sub R^n$ instead of just $R^n$.
 It was proved in \cite[Proposition 2.17]{eep2} that if $\cal C$ is a \emph{special} linear decomposition of $R^n$, then $\st_{\cal C}(C)$ is an open (usual) cell. This implied the conjecture for an open definable set $X$. Indeed, if $\cal C$ is chosen to partition $X$, then  the collection of all stars of  cells in $X$ provide a covering of $X$ by open cells. To check moreover that a cell $E$ is definably simply-connected,  observe first that by Berarducci-Fornasiero \cite[Lemma 3.2]{bf}, a bounded cell is definably contractible.
Then it is not hard to see that every definable loop in $E$ is contained in a bounded cell $B\sub E$ and hence it is definably homotopic (in $B$) to a constant loop (see Claim \ref{contractible1} below). 


In order to apply the same strategy and prove the conjecture for any definable set $X\sub R^n$ (not necessarily open), we consider the star $\st_{\cal C}(C)$ of a cell $C\sub X$ in a special linear decomposition $\cal C$ of $X$. For example, $\cal C$ could be the restriction to $X$ of a special linear decomposition $\cal C'$ of $R^n$ that partitions $X$. In that case,
$$\st_{\cal C}(C)= \st_{\cal C'}(C)\cap X.$$
The main obstacle now is that $\st_{\cal C}(C)$ need not be a cell anymore, and so it is far from clear whether it is contractible. For example, if we consider the definable contraction of $\st_{\cal C'}(C)$ from \cite[Lemma 3.2]{bf} (assuming  further it is bounded),  its restriction to $\st_{\cal C}(C)$ need not stay inside $\st_{\cal C}(C)$. As it turns out, proving the contractibility  of $\st_{\cal C}(C)$  amounts to proving the following general independent statement.\\



\noindent\textbf{Proposition 1.} \emph{Let $Y\sub R^n$ be a bounded definable set, $C\sub Y$ a linear cell, and $\cal D$ a special linear decomposition of $Y$ that contains $C$. Assume that
$$\forall D\in \cal D,\,\, C\cap cl(D)\ne \emptyset.$$
Then $Y$ is definably contractible.}\\

\noindent Proposition 1 is the heart of this paper and  is proved in Section \ref{sec-thms}. In the rest of this introduction we summarize its consequences and describe how it leads to the solution of the conjecture. We also sketch the main idea behind its proof. 

\smallskip
An immediate consequence of Proposition 1 is the following theorem.\\

\noindent\textbf{Theorem A.} \emph{Let $X\sub R^n$ be a definable set, $\cal C$ a special linear decomposition of $X$, and $C\in \cal C$ with $C\sub X$. If $\st_{\cal C}(C)$ is bounded, then $\st_{\cal C}(C)$ is definably contractible.}

\begin{proof}
Let $\cal D=\{D\in\cal C: C\cap cl(D)\ne\emptyset\}.$ By Proposition 1, $\st_{\cal C}(C) = \cup \cal D$ is definably contractible.
\end{proof}


\noindent In order to establish the conjecture, we need to get rid of the boundedness assumption, much alike we did in the case of an open set $X$. This is achieved in Lemma \ref{formain2}, which, together with Proposition 1, implies: \\

\noindent\textbf{Proposition 2.} \emph{Let $Y\sub R^n$ be a definable set, $C\sub Y$ a linear cell, and $\cal D$ a special linear decomposition of $Y$ that contains $C$. Assume that
$$\forall D\in \cal D,\,\, C\cap cl(D)\ne \emptyset.$$
Then $Y$ is definably simply-connected.}\\

\noindent Proposition 2  implies our second theorem.

$ $\\
\noindent\textbf{Theorem B.} \emph{ Let $X$, $\cal C$ and $C$ be as in Theorem A. Then $\st_{\cal C}(C)$ is definably simply-connected.}

\begin{proof}
Let $\cal D=\{D\in\cal C: C\cap cl(D)\ne\emptyset\}.$ By Proposition 2, $\st_{\cal C}(C) = \cup \cal D$ is definably simply-connected.
\end{proof}

\noindent As a corollary, we settle the conjecture in the pure vector space case:\\

\noindent\textbf{Corollary.}
\begin{enumerate}
   \item \emph{Every definable set is a finite union of relatively open definable subsets which are definably simply-connected.}

\item \emph{Every bounded definable set is a finite union of relatively open definable subsets which are definably contractible.}
 \end{enumerate}

\begin{proof}
Let $X$ be a definable set,  $\cal C'$  a special linear decomposition of $R^n$ partitioning it, and $\cal C$ its restriction to $X$. Then $X$ is clearly the finite union of all stars $\st_{\cal C}(C)=\st_{\cal C'}(C)\cap X$, for $C\in \cal C$. Since each $\st_{\cal C'}(C)$ is open (\cite[Lemma 2.16]{eep2}), $\st_{\cal C}(C)$ is relatively open in $X$. It is also definably simply-connected, by Theorem B. If, moreover, $X$ is bounded, then $\st_{\cal C}(C)$ is also definably contractible, by Theorem A.
\end{proof}

\noindent\emph{The main idea of the proof of Proposition 1.} The next  two paragraphs involve notions that are defined in Section \ref{preliminaries}, and their reading may be postponed until later.
Our strategy is to construct, for each $D\in \cal D$, a \emph{canonical retraction} of $cl(D)$ to the closure of the \emph{half-cell} $C'\sub C$. The virtue of such a retraction is two-fold. First, its restriction to $C\cup D$ is a deformation retraction of $C\cup D$ to $C'$ (Lemma \ref{Hinside} and Claim \ref{Hinside2}). Second, if $E$ is  another cell in $\cal D$, contained in the boundary of $D$, then the canonical retraction of $cl(D)$ to $cl(C')$ extends that of $cl(E)$ to $cl(C')$ (Lemma \ref{extend}). As a consequence, we can combine the above  retractions together
and obtain a deformation retraction of $Y$ to $C'$ (Proposition \ref{main1}). Finally, we observe that $C'$ is definably contractible (Lemma \ref{contract}).


The canonical retraction is in fact given relative to a \emph{corner} $c$ of $C$. To simplify the presentation, we first define it  for a \emph{canonical} linear cell $D$, a \emph{face} $C$ of $D$ and $c=0$ (Definition \ref{def-cancontr}). For an arbitrary bounded linear cell, the construction is deferred to Definition \ref{def-cangen}. Definition \ref{def-cancontr} is by recursion on $n$ and runs in parallel with Claim \ref{indeed}, where we prove that, at the recursive step, the resulting map $H_n$ is indeed a deformation retraction with the required properties. The definition is rather intricate and to facilitate its reading we illustrate it with Example \ref{ex1}.
The choice of our construction, and especially of retracting $Y$ to $C'$ as opposed to $C$, is given an explanation in Remark \ref{explainC'}.\\

\noindent\emph{Structure of the paper.} In Section \ref{preliminaries}, we introduce our terminology and prove some basic facts.
In Section \ref{sec-cancontr}, we give the construction of a canonical retraction.  In Section \ref{sec-thms}, we conclude the proofs of Propositions 1 and 2.\\


\noindent\emph{Acknowledgements.} I thank the referee for suggesting many improvements on the original manuscript.

\section{Preliminaries}\label{preliminaries}

Let us first fix some notation of this paper.
By $0$ we denote the origin of the space at hand. We let $R^{0}=\{0\}$. We also denote by $0:X\to R$ the map $0(x)=0$, whereas by $1_X$ we denote the identity map on $X$. We  write $[a, a]$ for $\{a\}$, and $Im(f)$ and $\Gamma(f)$ for the image and graph, respectively, of a function $f$. By a \emph{box} we mean a bounded set of the form
$$B=(a_1, b_1)\times \dots \times (a_n, b_n),$$
where $a_i<b_i\in R$. If $m\le n$, then $\pi_m:R^n\to R^m$ denotes the projection onto the first $m$ coordinates. We write $\pi$ for $\pi_{n-1}$. If \cal C is a collection of sets in $R^n$, by $\pi_m(\cal C)$ we mean the collection of their projections on $R^m$. If $Y\sub R^n$ is a definable set, then the restriction of $\cal C$ to $Y$ is the collection of sets $\{C\cap Y:\, C\in \cal C\}$. If $\sigma=(j_1, \dots, j_n)$ and $\tau=(i_1, \dots, i_n)$ are in $\{0, 1\}^n$, then $\sigma\le \tau$ (respectively, $\sigma<\tau$) means that for every $m$, $j_m\le i_m$ (respectively, $j_m<i_m$). If $a\in R$ and $X\sub R$, then $a<X$ means that $a<x$ for all $x\in X$. We denote by $cl(X)$ the topological closure of a set $X$.

\subsection{Special linear decompositions and stars}

We recall some basics for semi-linear sets, revisit special linear decompositions and stars from \cite{eep2}, and prove a few simple facts. A  function $f: R^n\to  R$  is called  \emph{linear} (or  \emph{affine)}, if it is of the form
$$f(x_1, \ldots, x_n)=\lambda_1 x_1+\ldots+\lambda_n x_n+a,$$
where  $\lambda_i\in \Lambda$ and $a\in R$.  For a non-empty set $X\subseteq  R^n$, we denote by $L(X)$ the set of restrictions to $X$ of linear functions and by $L_\infty(X)$ the set $L(X)\cup\{\pm\infty\},$ where we regard $-\infty$ and $+\infty$ as constant functions on $X$. Obviously, if $f\in L(X)$ then it extends uniquely to a linear map on $cl(X)$, and hence we write $f(a)$ for its value at $a$, even if $a\not\in X$. We also write $f_{| Y}$ for its restriction to a set $Y$, and $\Gamma(f)_Y$ for the graph of $f_{| Y}$.
If $f, g\in L_\infty(X)$ with $f(x)<g(x)$ for all $x\in X$, we write $f<g$ and  denote
$$(f,g)_X=\{(x, y)\in X\times R: f(x)<y<g(x)\}.$$ The notations $[f, g)_X$, $(f, g]_X$ and $[f, g]_X$ obtain the obvious meanings. By \cite[Lemma 2.8]{eep2}, if $C=(f, g)_X$, then $\pi(cl(C))=cl(\pi(C))$. We use this fact repeatedly.

A \emph{linear cell} in $R^n$ is defined similarly to \cite[Chapter 3, (2.3)]{vdd-book}, recursively, as follows:

\begin{itemize}
 \item
 $C\sub R$ is a linear cell if it is a singleton, or an open interval with endpoints in $R\cup\{\pm\infty\}$.
 \item
 $C\sub R^n$, $n>1$, is a linear cell if it is a set of the form $\Gamma(f)$, for some $f\in L(X)$, or $(f,g)_X$, for some $f,g\in L_\infty(X)$, and $X$ is a linear cell in $R^n$.
\end{itemize}
We call $\pi(C)$ the \emph{domain}, and $f, g$ the \emph{cell-maps}, of $C$. We attach an index $(i_1, \dots, i_n)\in \{0, 1\}^n$ to each linear cell $C$, such that $i_m=0$ if and only if  $\pi_m(C)$ is the graph of a function, for $1\le m\le n$.

We refer the reader to \cite[Chapter 3, (2.10)]{vdd-book} for the definition of a \emph{decomposition} of $R^n$. A \emph{linear decomposition of $R^n$} is then a decomposition $\mathcal{C}$ of $R^n$ such that each $B\in \mathcal{C}$ is a linear cell. The \emph{linear cell decomposition theorem}  can be proved similarly to \cite[Chapter 3, (2.11)]{vdd-book} and has already been observed in \cite[Section 3]{elst}:\\

\noindent\textbf{Linear cell decomposition theorem.}
\begin{enumerate}
\item
{\em Given any definable sets $A_1, \dots, A_k\subseteq  R^n$, there is a linear decomposition $\mathcal{C}$ of $R^n$ that partitions each $A_i$.}

\item
{\em Given a definable function $f:A\to  R$, there is a linear decomposition $\mathcal{C}$ of $R^n$ that partitions $A$ such that the restriction $f_{\res B}$ to each $B\in\mathcal{C}$ with $B\subseteq  A$ is linear.}\smallskip
\end{enumerate}

The notion of a `special linear decomposition' was introduced in \cite{eep2} and we recall it here, in a slightly different version. First, let us  define a \emph{linear decomposition of a definable set $Y$} as the restriction to $Y$ of  a linear decomposition of $R^n$ that partitions $Y$. (Equivalently, one could follow \cite[Chapter 4 (2.5)]{vdd-book}, where simply the properties of a decomposition of $R^n$ are required for a partition of $Y$.)

\begin{defn}\label{def-special} A \emph{special linear decomposition} of a definable set $Y\sub R^n$ is defined recursively on $n$, as follows. Any linear decomposition of $Y\sub R$ is special. A linear decomposition $\cal C$ of $Y$, $n>1$, is special if:
\begin{enumerate}
  \item $\pi(\cal C)$ is a special linear decomposition of $\pi(Y)$.

\item  For every two cells $\Gamma(f)_S$ and $\Gamma(g)_T$ in $\cal C$ with $S\sub cl(T)$,
$$f_{|S}<g_{|S} \text{ or } f_{|S}=g_{|S} \text{ or } f_{|S}>g_{|S}.$$

\item For every two cells  $(f,g)_T$ and $X$ in \cal C, where $X=\Gamma (h)_S, (h,k)_S$ or $(k, h)_S$,
$$ \text{ there is no $c\in cl(S)\cap cl(T)$ such that $f(c)<h(c)<g(c)$.}$$
\end{enumerate}
\end{defn}

The above definition differs from \cite[Definition 2.5]{eep2} in that (a) it is given for any set $Y$ and not just $R^n$, and (b) it further requires property (2). Observe that (2) for $Y=R^n$ was proved in \cite[Lemma 2.12]{eep2}, but that does not guarantee it for any $Y\sub R^n$. Having it for any $Y$ will be handy when stating Lemma \ref{formain2} below.  Meanwhile, none of (a) or (b) above causes serious diverging from \cite{eep2}, since  the proof of \cite[Lemma 2.6]{eep2} actually shows:

\begin{fact}\label{fact-special1} Let $Y\sub R^n$ be a definable set. Then for any linear decomposition $\mathcal{D}$ of $Y$, there is a special linear decomposition $\mathcal{C}$ of $Y$ that refines $\mathcal{D}$ (that is, every cell in $\mathcal{D}$ is a union of  cells in $\mathcal{C}$).
\end{fact}

\noindent We include the proof of the above fact in the Appendix, for completeness.
Let us recall here another important corollary from \cite{eep2}, which we will use in the proof of Proposition \ref{main1}.

\begin{fact}\label{fact-special2}
Let $Y\sub R^n$ be a definable set,  $\mathcal{C}$ a special linear decomposition of $Y$, and $D, E\in \mathcal{C}$ such that $D\cap cl(E)\ne\emptyset$. Then $D\subseteq  cl(E)$.
\end{fact}
\begin{proof}
The proof in \cite[Corollary 2.15]{eep2}  uses
\cite[Lemma 2.14]{eep2}. Both references hold with $Y$ in place of $R^n$, with identical proofs, after replacing $R^n$ by $Y$, and \cite[Lemma 2.12]{eep2} by  Definition \ref{def-special}(2).
\end{proof}

\begin{remark}\label{rem1} If $\cal C'$ is a special linear decomposition of $R^n$ that partitions a definable set $X$, then its restriction $\cal C$ to $X$ is clearly a special linear decomposition of $X$. In fact, every special linear decomposition \cal C of $X$ can be obtained in this way, but we do not prove or make use of this fact here.
\end{remark}

\subsection{Canonical linear cells, faces and half-cells} The canonical retraction in Section \ref{sec-cancontr} is a deformation retraction of a \emph{canonical} linear cell to the closure of the \emph{half-cell} of one of its \emph{faces}. In this section we introduce these three notions.


\begin{defn}
Let $D\sub R^n$ be a linear cell. For every $i=1, \dots, n$, let $h_i$ be either
\begin{itemize}
  \item the unique linear map $h_i:\pi_{i-1}(D)\to R$ with $\pi_i(D)=\Gamma(h_i)$, or
  \item the unique pair of linear maps $h_i=(f_i, g_i)$ with $\pi_i(D)=(f_i, g_i)_{\pi_{i-1}(D)}$.
\end{itemize}
We call $h_1, \dots, h_n$ the \emph{defining maps of $D$}.
We call $D$ a \emph{canonical} linear cell if it is bounded, and for every $i$, $h_i=0$ or $h_i=(0, g_i)$.
\end{defn}

Note: since $R^0=\{0\}$, a canonical linear cell in $R$ is either $\{0\}$ or an interval $(0, a)$, $a\in R$. Also, for every canonical linear cell $D$, $0\in cl(D)$.

\begin{defn}[Faces]
Let $C, D\sub R^n$ be two canonical linear cells. We define, recursively on $n$, that $C$ is a \emph{face} of $D$, if $C$ is a $(j_1, \dots, j_n)$-cell and $D$ is a $(i_1, \dots, i_n)$-cell, such that
\begin{enumerate}
\item $(j_1, \dots, j_n)\le (i_1, \dots, i_n)$,
  \item $\pi(C)$ is a face of $\pi(D)$, and
  \item
 \begin{itemize}
  \item $C=\Gamma(0)_{\pi(C)}$, if $D=\Gamma(0)_{\pi(D)}$,
  \item $C=\Gamma(0)_{\pi(C)}$, if $D=(0, g)_{\pi(D)}$ and $j_n=0$,
  \item $C=(0, g)_{\pi(C)}$, if $D=(0, g)_{\pi(D)}$ and $j_n=1$.
\end{itemize}
\end{enumerate}
\end{defn}


We make a few easy observations. Assume $C$ is a face of $D$ as in the above definition, and let $\sigma=(j_1, \dots, j_n)$. We have that $C\sub cl(D)$. If $\sigma <(i_1, \dots, i_n)$, then $C$ is contained in the boundary of $D$. If  $\sigma=(0, \dots, 0)$, then $C=\{0\}$. If $\sigma=(i_1, \dots, i_n)$, then $C=D$. By induction on $n-m$, one can see that for all $m=1, \dots, n$, $\pi_m(C)\sub cl(\pi_m(D))$ is a face of $\pi_m(D)$. By induction on $n$, if $E$ is a face of $C$ and $C$ is a face of $D$, then $E$ is a face of $D$. 

 The following claim will be used later on.

\begin{claim}\label{bdry}
Let $\cal C$ be a special linear decomposition of a definable set, and $C, D\in \cal C$ two canonical linear cells with $C\sub cl(D)$. Then $C$ is a face of $D$.
\end{claim}
\begin{proof} Let $C\sub R^n$ be a $(j_1, \dots, j_n)$-cell and $D\sub R^n$ a $(i_1, \dots, i_n)$-cell. We prove the claim by induction on $n$. For $n=1$, it is immediate. For $n>1$, we have $\pi(C)\sub cl(\pi(D))$ and thus, by induction, $\pi(C)$ is a face of $\pi(D)$. If $j_n=0$, then it is immediate that $C$ is a face of $D$, so assume $i_n=j_n=1$. Let $D=(0, g)_{\pi(D)}$ and $C=(0, k)_{\pi(C)}$, with $\pi(C)\sub cl(\pi(D))$. We want to prove that $k=g_{| \pi(C)}$. Since $cl(C)\sub cl(D)$, we know that for every $t\in \pi(C)$,
$$0<k(t)\le g(t),$$
and by Definition \ref{def-special}(3), the last inequality cannot be strict.
\end{proof}

We now proceed to the notion of a half-cell.

\begin{defn}
Let $A\sub R^{n-1}$ and  $g\in L(A)$ with $g>0$. The \emph{half-map} of $g$ is the map $F\in L(A)$ given by
$$F(x)=\frac{g(x)}{2}.$$
It is clear that $0<F_{|A}<g_{|A}$ and $0\le F_{|cl(A)}\le g_{|cl(A)}$. \smallskip

\noindent Now let $C$ be a canonical linear cell. We define the \emph{half-cell} of $C$, denoted simply by $C'$, recursively, as follows:
\begin{enumerate}
  \item $n=1$. If $C$ is a singleton, then $C'=C$. If $C=(0, a)$, then $C'=(0, \frac{a}{2}]$.
  \item $n>1$. Let $A=\pi(C)$ and $A'$ its half-cell.
  \begin{itemize}
  \item If $C=\Gamma(0)_A$, then $C'=\Gamma(0)_{A'}$.
  \item If $C=(0, g)_A$, then $C'=(0, F]_{A'}$, where $F$ is the half-map of $g$.
  \end{itemize}
\end{enumerate}
\end{defn}

By construction, the half-cell of $\pi(C)$ equals $\pi(C')$.

\begin{lemma}\label{ff'}
Let $C=(0, g)_B$ and $D=(0, h)_{A}$ be two canonical linear cells, such that $C$ is a face of $D$  (and so $B$ is a face of $A$). Let $f\in L(B)$ be the half-map of $g$ and $e\in L(A)$  the half-map of $h$. Then:
$$f = e_{| B}.$$
\end{lemma}
\begin{proof} Clear from the definition.
\end{proof}

It is also clear that for $C$ and $D$ as above, $C'$ is a face of $D'$, but we will not make use of this fact here.

\subsection{Homotopy} We recall the definable analogues of standard notions from algebraic topology, and prepare the ground for the construction of a canonical retraction in Section \ref{sec-cancontr}.

\begin{defn}\label{def-retraction}
Let $A\sub X$ be two definable sets. We say that \emph{$X$ deformation retracts to $A$} if there is a definable continuous $H:[0, q]\times X\to X$ such that:
\begin{enumerate}
\item $H(0, X)=A$
  \item $\forall t\in [0, q]$, $H(t, -)_{\res A}= 1_A$
  \item $H(q, -)= 1_X$.
\end{enumerate}
We call $H$  a \emph{deformation retraction of $X$ to $A$}. If $A$ above is a singleton $\{c\}$, we  say that $X$ is  \emph{definably contractible (to $c$)}, and that $H$ is a \emph{definable contraction} of $X$ to $c$.
\end{defn}

Note that the above notion of a deformation retraction is often regarded as a `strong' one  in the literature, because of (2). Note also that we have omitted the word `definable' from our terminology, for simplicity.

\begin{defn}
Let $A\sub X\sub X'$ be three definable sets and suppose that $H:[0, q]\times X\to X$ and $H':[0, q']\times X'\to X'$ are  deformation retractions of $X$ and $X'$, respectively, to $A$. We say that \emph{$H'$ extends $H$} if $q\le q'$ and for every $(t, x)\in [0, q]\times X$,
$$H(t, x)=H'(t, x).$$
\end{defn}

By a \emph{definable path}  we simply mean a definable continuous map $\gamma:[0, p]\to R^n$. We call $\gamma$ a \emph{loop} if $\gamma(0)=\gamma(p)$. Given $c\in R^n$, the constant path $\ve_c$ is defined by $\ve_c(x)=c$ and its domain can vary according to  context. A definable set is called \emph{definably connected} if every two points of it are connected with a definable path.

\begin{defn} Let $X$ be a definable set and  $\gamma, \delta:[0, p]\to X$  two definable paths with $\gamma(0)=\delta(0)$ and $\gamma(p)=\delta(p)$. We say that $\gamma$ and $\delta$ are \emph{definably homotopic (in $X$)} if there is a definable continuous $F:[0, q]\times [0, p]\to X$ such that:
\begin{enumerate}
\item $F(0, -)=\gamma$.
  \item $\forall t\in [0, q]$, $F(t, 0)=\gamma(0)$ and $F(t, p)=\gamma(p)$.
  \item $F(q, -)= \delta$.
\end{enumerate}
We call $X$ \emph{definably simply-connected} if it is definably connected and every definable loop in it is definably homotopic to a constant path.
\end{defn}

\begin{claim}\label{contractible1}
 Suppose that $X$ is definably contractible. Then it is definably simply-connected.
\end{claim}
\begin{proof}
Given a deformation retraction $H:[0, q]\times X\to X$ of   $X$ to $c$, and a definable path  $\gamma:[0, p]\to X$, the map $F:[0, q]\times [0, p]\to X$ defined by
$$F(t, x)=H(t, \gamma(x))$$
witnesses that $\gamma$ is definably homotopic to the constant path $\ve_c$. Moreover, given any $x\in X$, the map $H(-, x)$ is a definable path from $c$ to $x$, witnessing that $X$ is also definably connected.
\end{proof}

By \cite[Lemma 3.2]{bf}, every bounded cell is definably contractible. We also know the converse. Although not used in this paper, we record it for completeness.

\begin{fact}\label{ex-bounded} An unbounded definable set is not definably contractible.
\end{fact}
\begin{proof}
Let $X$ be an unbounded definable set.  Suppose $H:[0, q]\times X \to X$ is a definable contraction of $X$ to a point $c\in X$. Consider the map $f:[0, q] \to R$ given by
$$t\mapsto \sup\{|x_1 +\dots + x_n|: (x_1, \dots, x_n)\in H(t, X)\}.$$
Then $f$ is a definable map  whose image contains an unbounded interval, since $X$ is unbounded. But that is a contradiction, because  $\cal R$ has no poles; that is, there are no definable bijections between bounded and unbounded sets (\cite{ed-str}).
\end{proof}

We note here that if $\cal R$ were to expand an ordered field, then unbounded cells could also be shown to be definably contractible. For example, $R$ itself would be definably contractible to $0$ via $H: [0,1]\times R\to R$ with $H(t, x)=tx$.
The lack of multiplication in our setting is of course one of the main particularities.

\begin{lemma}\label{contract}
Let $C$ be a canonical linear cell. Then its half-cell $C'$ is definably contractible.
\end{lemma}
\begin{proof}
By induction. Let $n=1$. If $C$ is a singleton, it is trivial, and if $C=(0, a)$, let $H:[0, \frac{a}{2}]\times C'\to C'$ given by $H(t, x)=\max \{\frac{a}{2}-t, x\}$. Then $H$ is a definable contraction of $C'$ to $\{\frac{a}{2}\}$.
Now let $n>1$. By induction, there is a definable contraction
$$H_1:[0, q_1]\times \pi(C')\to \pi(C')$$ of $\pi(C')$ to some $c\in \pi(C')$. If $C=\Gamma(0)_A$, let  $H:[0, q_1]\times C'\to C'$ with
$$H(t, x, y)=(H_1(t, x), 0).$$
Then $H$ is a definable contraction of $C'$ to $(c, 0)$.  If $C=(0, g)_A$, let $F$ be the half-map of $g$, and
 $$H:[0, q_1+\sup Im F]\times C'\to C'$$
with
 $$H(t, x, y)=\begin{cases}
   (H_1(t, x), F(H_1(t, x))), &\text{ if $t<q_1$},\\
   (x, \max \{F(x)-(t-q_1), y\}), & \text{ if $t\ge q_1$}.
 \end{cases}$$
It is easy to check that $H$ is a definable contraction of $C'$ to  $(c, F(c))$.
\end{proof}

\section{Canonical retractions}\label{sec-cancontr}

We are now ready to present the construction of a canonical retraction. As mentioned in the introduction, we first give it for canonical linear cells (Definition \ref{def-cancontr}) and then for arbitrary bounded linear cells (Definition \ref{def-cangen}); this significantly reduces the notational complexity of the presentation. Definition \ref{def-cancontr} is given recursively on $n$ and runs in parallel with Claim \ref{indeed}, where we verify  all necessary properties at the recursive step. Before stating the (rather lengthy) definition, we illustrate it with an example.

\begin{example}\label{ex1} Let $D=(0, g)_{(0, a)}$ be a canonical linear cell in $R^2$,  $f$  the half-map of $g$, $C$ a face of $D$, and $C'$ the half-cell of $C$. We illustrate Cases (II) and (III) of Definition \ref{def-cancontr} below, for $\pi(C)=\{0\}$. Let $q=a+\sup Im g$, and define the deformation retraction $$H:[0, q]\times cl(D)\to cl(D)$$ of $cl(D)$ to $cl(C')$, as follows.\\

\noindent Case (II): $C=\{0\}$. Then $C'=C$ and:
$$H(t, x, y)=\begin{cases}
   (t, \min \{y, t, f(t)\}), &\text{if $t<x$},\\
   (x, \min \{y, t, t-x+f(x)\}), &\text{if $t\ge x$}.
\end{cases}$$

\noindent Case (III): $C=\{0\}\times (0, g(0))$. Then $C'=\{0\}\times (0, g(0)/2)$ and:
$$H(t, x, y)=\begin{cases}
   (t, \min \{y, f(t)\}), &\text{if $t<x$},\\
   (x, \min \{y, t-x+f(x)\}), &\text{if $t\ge x$}.
\end{cases}
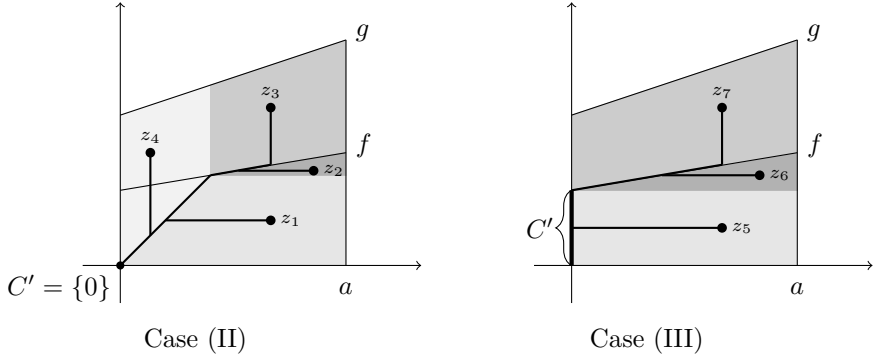
$$

\begin{tikzpicture}

\filldraw [black!10] (0,0) -- (1.2,1.2) -- (3,1.2) -- (3,0) -- (0,0);
\filldraw [black!5] (0,0) -- (1.2,1.2) -- (1.2,2.4) -- (0,2) -- (0,0);
\filldraw [black!20] (1.2,1.2) -- (3,1.5) -- (3,3) -- (1.2,2.4) -- (1.2,1.2);
\filldraw [black!30] (1.2,1.2) -- (3,1.2) -- (3,1.5) -- (1.2,1.2);

\draw[thick] (1,-1)  node {Case (II)};

\draw [->] (-.5,0) -- (4,0);
\draw [->] (0,-.5) -- (0,3.5);

\draw (3, 0) -- (3,3);
\draw (0, 2) -- (3,3);

\draw (0,1) -- (3, 1.5);
\draw[ thick] (1.2,1.2) -- (2,1.34);
\draw[ thick] (0,0) -- (1.2, 1.2);

\draw[thick] (0.4,0.4) -- (0.4, 1.5);
\draw[thick] (2,1.34) -- (2, 2.1);
\draw[thick] (0.6,0.6) -- (2, 0.6);
\draw[thick] (1.57,1.26) -- (2.57,1.26);

\draw[thick, fill]  (2,0.6) circle (.5mm)  node[right] {$\scriptstyle z_1$};
\draw[thick, fill]  (0.4,1.5) circle (.5mm) node[above] {$\scriptstyle z_4$};
\draw[thick, fill]  (2,2.1) circle (.5mm)  node[above] {$\scriptstyle z_3$};
\draw[thick, fill]  (2.57,1.26) circle (.5mm)  node[right] {$\scriptstyle{z_2}$};

\draw[thick, fill]  (3,-0.1) node[below] {$a$};;
\draw[thick, fill]  (3,1.6) node[right] {$f$};;
\draw[thick, fill]  (3,3.1) node[right] {$g$};;
\draw[thick, fill]  (0,-0.3) node[left] {$C'=\{0\}$};;
\draw[thick, fill]  (0,0) circle (.4mm);

\begin{scope}[xshift=6cm]

\filldraw [black!10] (0,0) -- (3,0) -- (3,1) -- (0,1);
\filldraw [black!30] (0,1) -- (3,1) -- (3,1.5) -- (0,1);
\filldraw [black!20] (0,1) -- (3,1.5) -- (3,3) -- (0,2) -- (0,1);

\draw[thick] (1,-1)  node {Case (III)};

\draw [->] (-.5,0) -- (4,0);
\draw [->] (0,-.5) -- (0,3.5);

\draw (3, 0) -- (3,3);
\draw (0, 2) -- (3,3);

\draw (0,1) -- (3, 1.5);
\draw[ thick] (0,1) -- (2, 1.34);

\draw[ thick] (0,0.5) -- (2, 0.5);
\draw[ thick] (1.2,1.2) -- (2.5, 1.2);
\draw[ thick] (2,1.34) -- (2, 2.1);

\draw[thick, fill]  (2,0.5) circle (.5mm)  node[right] {$\scriptstyle z_5$};
\draw[thick, fill]  (2,2.1) circle (.5mm)  node[above] {$\scriptstyle z_7$};
\draw[thick, fill]  (2.5,1.2) circle (.5mm)  node[right] {$\scriptstyle z_6$};

\draw[thick, fill]  (3,-0.1) node[below] {$a$};;
\draw[thick, fill]  (3,1.6) node[right] {$f$};;
\draw[thick, fill]  (3,3.1) node[right] {$g$};;
\draw[thick, fill]  (-.1,0.5) node[left] {$C'$};;

\draw [decorate,decoration={brace,amplitude=6pt}] (0,0) -- (0,1);


\draw[ultra thick] (0,0) -- (0,1);

\end{scope}

\end{tikzpicture}

\noindent The above pictures depict the images of $H(-, z_i)$, for various $z_i=(x, y)\in D$. Depending on the location of $z_i$, the map $t\mapsto H(t, z_i)_2$ takes the following values:\vskip.2cm

\begin{tikzpicture}

\draw[thick, fill]  (1,4.5) node[right] {\hspace*{1cm}$z_1 :\,\, t,\,\, y$};

\draw[thick, fill]  (1,4) node[right] {\hspace*{1cm}$z_2: \,\,  t,\,\,  f(t),\,\,  y$};

\draw[thick, fill]  (1,3.5) node[right] {\hspace*{1cm}$z_3: \,\,  t,\,\,  f(t),\,\,  t-x+f(x),\,\, y$};

\draw[thick, fill]  (1,3) node[right] {\hspace*{1cm}$z_4: \,\,  t,\,\,  y$};

\begin{scope}[xshift=7cm]

\draw[thick, fill]  (1,4.5) node[right] {$z_5:  \,\, y$};

\draw[thick, fill]  (1,4) node[right] {$z_6:\,\,  f(t),\,\,  y$};

\draw[thick, fill]  (1,3.5) node[right] {$z_7:\,\,   f(t),\,\,  t-x+f(x),\,\, y$};

\end{scope}
\end{tikzpicture}
\end{example}

Our canonical retractions  have the  extra property of being `controlled' in the following sense.

 \begin{defn} Let $H:[0, q]\times X\to X$  be a deformation retraction of $X$ to $A\sub X$. For every $x\in X$, the \emph{fixing point of $x$ under $H$} is the point $\alpha_x\in [0,q]$ given by
$$\alpha_x=\min\{t\in [0, q]: H(t, x)=x\},$$
which exists by continuity of  $H$. We call $H$ \emph{controlled}, if for every $x\in X$ and  $t\ge \alpha_x$, $H(t, x)=x$.
\end{defn}

Note that there may be more ways to handle Example \ref{ex1}. However,  the suggested retractions can generalize to an arbitrary $n$, as follows.


\begin{defn}[Canonical retraction]\label{def-cancontr}
Let $D\sub R^n$ be a canonical linear $(i_1, \dots, i_n)$-cell,   $C$ a canonical linear  $(j_1, \dots, j_n)$-cell, and assume that $C$ is a face of $D$. Let $C'$ be the half-cell of $C$. The \emph{canonical retraction $H$ of $cl(D)$ to $cl(C')$},
$$H_n : [0, q_n]\times cl(D)\to cl(D),$$
 is a controlled retraction defined recursively on $n$, as follows. Let $h_1, \dots, h_n$ be the defining maps of $D$. So $h_i=0$ or $h_i=(0, g_i)$.\\

\noindent $n=1$.  
First, define
$$q_1=\begin{cases}
  0, & \text{if $i_1=0$},\\
  g_1(0), & \text{if $i_1=1$}.
 \end{cases}.$$
 Now let $y\in cl(D)$. If $i_1=j_1$, then define $H_1(t, y)=\min\left\{y, t+\frac{q_1}{2}\right\}$. If $i_1>j_1$, define
 $$H_1(t, y) =\min \{t, y\}.$$\vskip.2cm


\noindent $n>1$. Let $H_{n-1}:[0, q_{n-1}]\times cl(\pi(D))\to cl(\pi(D))$ be the canonical retraction of $cl(\pi(D))$ to $cl(\pi(C'))$. If $i_n=1$, we let $f: \pi(D)\to R$ be the half-map of $g_n$, and so $C'=(0, f]_{\pi(C')}$.
Recall that for every $x\in cl(\pi(D))$, $\alpha_x$ denotes the fixing point of $x$ under $H_{n-1}$,
 $$\alpha_x=\min\{t\in [0, q_{n-1}]:\, H_{n-1}(t, x)= x\}.$$
Now let
$$q_n=
\begin{cases}
q_{n-1}, &\text{if $i_n=0$},\\
q_{n-1}+ \sup Im g_n, & \text{if $i_n=1$},
\end{cases}
$$
and for every $(x, y)\in cl(D)$, define $H_n(t, x, y)$ by cases, as follows.\\

\noindent (I) If $i_n=j_n=0$, then
$H_n(t, x, y)=(H_{n-1}(t, x), 0)$.\\

\noindent (II) If $i_n>j_n$, then
$$H_n(t, x, y)=\begin{cases}
   (H_{n-1}(t, x), \min \{y, t, f H_{n-1}(t, x)\}), &\text{if $t< \alpha_x$},\\
   (H_{n-1}(t, x), \min \{y, t, t-\alpha_x+f H_{n-1}(t, x)\}), &\text{if $t\ge \alpha_x$}.
\end{cases}$$

\noindent (III) If $i_n=j_n=1$,  then $$H_n(t, x,y)=\begin{cases}
(H_{n-1}(t, x), \min \{y, f H_{n-1}(t, x)\}), &\text{if $ t< \alpha_x$},\\
   (H_{n-1}(t, x), \min \{y,  t-\alpha_x+f H_{n-1}(t, x)\}), &\text{if $ t\ge\alpha_x$}.
\end{cases}$$
\end{defn}

\smallskip
\noindent\textbf{Note:} In Cases (II) and (III), for $t\ge \alpha_x$, $H_{n-1}(t,x)=x$, since $H_{n-1}$ is controlled.\\

We next verify that, at the recursive step,  $H_n$ has the required properties. The proof is rather straightforward, but we include it for completeness.

\begin{claim}\label{indeed}  $H_n$ is a controlled deformation retraction of $cl(D)$ to $cl(C')$.\end{claim}

\begin{proof} We work by induction. For $n=1$, all properties below are immediate and we omit their proofs. Moreover, $\alpha_y$ is continuous in this case. Let $n>1$.  We denote $g=g_n$. \vskip.2cm

\noindent  \textbf{(a) $H_n$ is a map $H_n:[0, q_n] \times  cl(D)\to cl(D)$.}\smallskip

Let $t\in [0, q_n]$ and $(x, y)\in cl(D)$. In all Cases (I) - (III), the first $n-1$ coordinates of $H_n(t, x, y)$ equal $H_{n-1}(t,x)\in cl(\pi(D))$, by induction. So in Case (I), we are done. For Cases (II) and (III), we  need to check that
$$H_n(t, x,y)_n \le   g H_{n-1}(t,x),$$
In both Cases (II) and (III), if $t< \alpha_x$, we have
$$H_n(t, x, y)_n\le f H_{n-1}(t, x)\le g H_{n-1}(t, x),$$
and if $t\ge\alpha_x$, then
$$H_n(t, x,y)_n \le y\le g(x)= g H_{n-1}(t,x).$$

\vskip.2cm

\noindent   \textbf{(b) $H_n$ is a deformation retraction of $cl(D)$ to $cl(C')$.}\smallskip

By induction, it is easy to verify that $H_n$ is  definable.
We  prove that $H_n$ is continuous. Clearly, by  definition of $H_n$ and induction, it suffices to check that the map $(x, y)\mapsto \alpha_{(x, y)}$ is continuous. We do that by induction as well, where for $n=1$, we have already observed that $\alpha_y$ is continuous. Now,  in Case (I) clearly $\alpha_{(x,y)}=\alpha_x$, so it is continuous. In Cases (II) and (III),  it is easy to see that
$$\alpha_{(x,y)}= \max \{ \alpha_x, \min\{t\in [\alpha_x, q_n]: H_n(t,x,y)_n=y\}\}.$$
Since $\alpha_x$ is continuous and $H_{n-1}$ is controlled, for Case (II) it suffices to check that for any $t>\alpha_x$, the map $$y\mapsto  \min\{t\in [\alpha_x, q_n]: \min\{y, t, t-x+f(x)=y\}\}$$
is continuous. But this is clear. Similarly for Case (III).

\vskip.2cm
We  now verify the three properties of Definition \ref{def-retraction}.

\vskip.2cm\noindent (1) We prove that for every $(x, y)\in cl(D)$, $H_n(0, x, y)\in cl(C')$. In all Cases (I) - (III), the first $n-1$ coordinates of $H_n(0, x, y)$ equal $H_{n-1}(0, x)\in cl(\pi(C'))$, by induction. In Case (I), we are clearly done. In Cases (II) and (III), we only need to check that
$$H_n(0, x, y)_n\le f H_{n-1}(0, x),$$
which is clear from their definition.


\vskip.2cm\noindent (2) We prove that for every $(x, y)\in cl(C')$ and $t\in [0, q_n]$, $H_n(t, x, y)=(x, y)$. In all Cases (I) - (III), the first $n-1$ coordinates of $H_n(t, x, y)$ are $H_{n-1}(t, x)=x$, by induction. Moreover, $\alpha_x=0$. So we only need to prove that
$$H_n(t, x, y)_n=y.$$
In Case (I), it is clear. In Case (II), $(x, y)\in cl(C')$ implies that $y=0$, and hence $H_n(t, x, y)_n=y=0$. In Case (III), $(x, y)\in cl(C')$ implies  $y\le f(x)$, and since $\alpha_x=0$, we have $$y\le f(x) =fH_{n-1}(t, x).$$

\vskip.2cm\noindent (3) We prove that for every $(x, y)\in cl(D)$, $H_n(q_n, x, y)=(x, y)$. By induction, $H_{n-1}(q_{n-1}, x)=x$, and hence $q_n\ge q_{n-1}\ge \alpha_x$. Therefore, in all Cases (I) - (III), the first $n-1$ coordinates of $H_n(q_n, x, y)$ equal $H_{n-1}(q_n, x)=x$.
So we only need to prove
$$H_n(q_n, x, y)_n=y.$$
Case (I) is clear, whereas for Cases (II) and (III), we need to show that
$$y\le q_n\,\, \text{ and }\,\, y\le q_n-\alpha_x+f(x).$$ But in both cases, we have:
$$q_n\ge \sup Im g\ge y,$$
and
$$q_n-\alpha_x+f(x)=q_{n-1}+\sup Im g - \alpha_x + f(x)\ge \sup Im g +f(x)\ge y.$$

\vskip.2cm

\noindent  \textbf{(c) $H_n$ is controlled.}\smallskip

We prove that for every $(x, y)\in cl(D)$ and $t\ge \alpha_{(x, y)}$, $H_n(t, x, y)=(x, y)$. First observe that for $(x, y)\in cl(D)$, $\alpha_x\le \alpha_{(x, y)}$. Indeed, since $H_n(\alpha_{(x, y)}, x, y)=(x, y)$, we have $H_{n-1}(\alpha_{(x, y)},x)=x$ and hence $\alpha_x\le \alpha_{(x, y)}$. Now, Case (I) is clear, whereas for (II) and (III), we need to prove that for $t\ge \alpha_{(x, y)}$,
$$H_n(t, x, y)_n=y.$$
To this end, observe that for $t\ge \alpha_{(x,y)}\ge \alpha_x$, $f H_{n-1}(t, x)=f(x)$ is fixed. Hence each of
$$t,\,\, f H_{n-1}(t, x),\,\, t-\alpha_x+f H_{n-1}(t, x)$$ is  increasing for $t\ge \alpha_{(x, y)}$. So since $y$ is smaller or equal than some of them at $t=\alpha_{(x,y)}$ then so it is for $t\ge \alpha_{(x, y)}$. By definition of Cases (II) and (III), we are done.
\end{proof}

In the next two lemmas, if $H$ is the canonical retraction of $cl(D)$ to $cl(C)$, then $H_1$ denotes the canonical retraction of $cl(\pi(D))$ to $cl(\pi(C))$.

\begin{lemma}\label{Hinside} Let $D$ be a canonical linear cell and $C$ one of its faces. If $H$ is a canonical retraction of $cl(D)$ to $cl(C')$, then $H_{\res [0, q]\times (C\cup D)}$ is a deformation retraction of $C\cup D$ to $C'$.
\end{lemma}
\begin{proof} The proof resembles that of Claim \ref{indeed}(a). For $n=1$, it is  immediate.  Let $n>1$, $t\in [0, q_n]$ and $(x, y)\in C\cup D$.  We need to check that $H(t, x, y)\in C\cup D$.
In all Cases (I) - (III), the first $n-1$ coordinates of $H(t, x, y)$ equal $H_1(t,x)\in \pi(C\cup D)$, by induction. So in Case (I), we are done. For Cases (II) and (III), we  need to check that
$$H(t, x,y)_n <   g H_1(t,x),$$
and in Case (III), we  need moreover  $0<H(t, x,y)_n$.
The latter is clear since in Case (III), $(x, y)\in C\cup D$ implies that both $y$ and $f H_1(t, x)$ are positive. For the former, in both Cases (II) and (III), if $t< \alpha_x$, we
have
$$H(t, x, y)_n\le f H_1(t, x)\le g H_1(t, x).$$
If the last inequality if strict, we are done. Assume  $f H_1(t, x)= g H_1(t, x)$. By the definition of half-maps, this can only happen if $H_1(t, x)\in \pi(C)$ and $f_{|\pi(C)}=g_{|\pi(C)}=0$. However, that would imply $j_n=0$, and hence
$$H(t, x, y)=(H_1(t,x), 0)\in C.$$
If $t\ge\alpha_x$, then
$$H(t, x,y)_n \le y\le g(x)= g H_1(t,x),$$
Again, if the second inequality is strict, we are done. On the other hand, the equation $y=g(x)$ can only happen if $x\in \pi(C)$ and $y=0$. But that  would imply $j_n=0$, and hence again
$$H(t, x, y)=(H_1(t,x), 0)\in C.$$
\end{proof}

%

\begin{lemma}\label{extend} Let $C, E, D\sub R^n$ be  three canonical linear cells, and assume that $C$ is a face of $E$, and $E$ is a face of $D$. Let $H$ and $H'$ be the canonical retractions of $cl(E)$ and $cl(D)$ to $cl(C')$, respectively. Then $H'$ extends $H$.
\end{lemma}
\begin{proof} For $n=1$, this is immediate. Assume $D$ is a $(i_1, \dots, i_n)$-cell and $E$ a $(j_1, \dots, j_n)$-cell, $n>1$. Let $[0, q]$ and $[0, q']$ be the parameter sets of $H$ and $H'$, respectively. By induction, it is easy to see that $q'\ge q$. Let $t\in [0, q]$ and $(x, y)\in cl(E)$. We need to prove that $H(t, x, y)=H'(t, x, y)$.
By induction, $H_1(t, x)= H'_1(t, x)$. Hence, we only have to show
$$H(t, x, y)_n=H'(t, x, y)_n.$$

If $j_n=0$, then $y=0$, and clearly $H(t, x, y)_n=0$, by Case (I) of Definition \ref{def-cancontr}, whereas $H'(t, x, y)_n=0$, by Cases  (I) - (III).

So let $i_n=j_n=1$. Since $H'_1$ extends $H_1$, the corresponding fixing points $\alpha_x$ and $\alpha'_x$  coincide, whereas by Lemma \ref{ff'}, so do the half-maps of $g$ and $g'$ on $cl(\pi(E))$. It follows immediately that $H(t, x, y)_n=H'(t, x, y)_n$.
\end{proof}

\begin{remark}\label{explainC'}
It is possible to define a canonical retraction of $cl(D)$ to $cl(C)$, as opposed to $cl(C')$. However, Lemma \ref{Hinside} then becomes more difficult to achieve. Indeed, if we resembled Definition \ref{def-cancontr}, we would first need to replace the notion of a half-map by some suitable map which equals $g$ on $\pi(C)$. Then the resulting canonical retraction, restricted to $C\cup D$, would give a retraction of $C\cup D$ to $(0, g]_{\pi(C)}$ instead of $C=(0, g)_{\pi(C)}$.  To overcome this issue, one needs to give a  more elaborate definition of a canonical retraction, which we avoided doing here. We note that our canonical retraction is \emph{not} the concatenation of two retractions, one from $cl(D)$ to $cl(D')$, and then from $cl(D')$ to $cl(C')$.
\end{remark}



\subsection{Arbitrary linear cells}

We  now  extend the definition of canonical retractions to arbitrary bounded linear cells. The idea is simply to first map each such cell $D$ to a canonical linear cell $T(D)$, such that if $C\sub cl(D)$ is another linear cell and $c$ is a common `corner' of $C$ and $D$, then $T(c)$ becomes the origin, and $T(C)$ a face of $T(D)$. We then pullback the canonical retraction of $cl(T(D))$ to $cl(T(C)')$, to a deformation retraction of $cl(D)$ to $cl(T^{-1}(T(C)'))$, where $T(C)'$ is the half-cell of $T(C)$.

\begin{defn}[Corners of a linear cell]
Let $D\sub R^n$ be a linear $(i_1, \dots, i_n)$-cell. We define, recursively on $n$, the set of \emph{corners of $D$}. Let $(l_1, \dots, l_n)\le (i_1, \dots, i_n)$. 
\begin{enumerate}
  \item A point $c\in R$ is a $(l_1)$-corner of $D$ if
  \begin{itemize}
    \item $D$ is the singleton $\{c\}$, or
    \item if $D$ is an interval, and $c$ is the left endpoint, if $l_1=0$, and the right endpoint, if $l_1=1$.
  \end{itemize}

   \item A point $c\in R^n$ is a \emph{$(l_1, \dots, l_n)$-corner of $D$} if $a=\pi(c)$ is a $(l_1, \dots, l_{n-1})$-corner of $A=\pi(D)$ and
       \begin{itemize}
         \item $D=\Gamma(f)_A$ and $c=(a, f(a))$, or
         \item $D=(f, g)_A$, and $c=(a, f(a))$, if $l_n=0$, and $c=(a, g(a))$ if $l_n=1$.
       \end{itemize}
\end{enumerate}
A \emph{corner of $D$} is a $(l_1, \dots, l_n)$-corner for some $(l_1, \dots, l_n)$.
\end{defn}
Observe that if $D=(f, g)_A$, as above, with $f(a)=g(a)$, then its $(l_1, \dots, l_{n-1}, 0)$-corner and $(l_1, \dots, l_{n-1}, 1)$-corner coincide.

\begin{lemma}\label{specialcorner}
 Let $\cal C$ be a special linear decomposition of a definable set in $R^n$, $C, D\in \cal C$ two linear cells, and $c$ a corner of $C$. If $C\cap cl(D)\ne \emptyset$, then $c$ is also a corner of $D$.
\end{lemma}

\begin{proof}
By Fact \ref{fact-special2}, we know that $C\sub cl(D)$. We work by induction on $n$. For $n=1$, it is immediate. For $n>1$, let $A=\pi(D)$ and observe by induction that $a=\pi(c)$ is a corner of $A=\pi(D)$. For $D=\Gamma(f)_A$, it is then clear that $(a, f(a))$ is a corner of $D$. Now let $D=(f, g)_A$. Since $C\sub cl(D)$, we have two  cases: \smallskip


\noindent \textbf{Case I.} $C=\Gamma(h)_B\sub (f, g)_A$, and $B\sub cl(A)$. Since \cal C is special (Definition \ref{def-special}(3)), we have $h=f_{|B}$ or $h=g_{|B}$, and hence $c=(a, f(a))$ or $c=(a, g(a))$, respectively, which are both corners of $D$.\smallskip

\noindent \textbf{Case II.} $C=(k, h)_B\sub (f, g)_A$. Again, by Definition \ref{def-special}(3), we must also have $k=f_{|B}$ and $h=g_{|B}$, and hence $c=(a, f(a))$ or $c=(a, g(a))$, which are both corners of $D$.
\end{proof}

\begin{defn}[Canonical transformation of a bounded linear cell]
Let $D\sub R^n$ be a linear $(i_1, \dots, i_n)$-cell and $c$ its $(l_1, \dots, l_n)$-corner.
The \emph{canonical transformation $T_{D, c}$ associated to $D$ and $c$} is a linear map $T=T_{D, c}: cl(D)\to R^n$, defined recursively as follows.
\begin{enumerate}
  \item For $n=1$,

  \begin{itemize}
    \item if $D=\{c\}$, then $T(c)=0$.
  \item if $D$ is an interval, then $T(x)= |x- c|$.
  \end{itemize}

  \item For $n>1$, let $A=\pi(D)$ and $a=\pi(c)$. 

  \begin{itemize}
    \item  If $D=\Gamma(f)_A$, then $T(x, f(x))=(T_{A, a}(x), 0)$.
    \item If $D=(f, g)_A$, then
    $$T(x, t)=
    \begin{cases}
      (T_{A, a}(x), t-f(x)), & \text{ if $l_n=0$},\\
      (T_{A, a}(x), g(x)-t), & \text{ if $l_n=1$}.
    \end{cases}
    $$
  \end{itemize}
\end{enumerate}
We call $D_c=T_{D, c}(D)$ the \emph{canonical transformation of $D$ with respect to $c$}. 
\end{defn}

\begin{remark}\label{canret2} It is straightforward to check that:
\begin{enumerate}
  \item $D_c$ is a canonical linear $(i_1, \dots, i_n)$-cell.
  \item Let $\cal C$ be a special linear decomposition of a definable set, and $C, D\in \cal C$ two bounded cells. If $cl(C)\sub cl(D)$ and $c$ is a common corner of $C$ and $D$, then $T_{D,c}$ agrees with $T_{C, c}$ on $cl(C)$. Moreover, $C_c, D_c$ belong to a special linear decomposition of a definable set (their union), and $C_c\sub cl(D_c)$. Hence, by Claim \ref{bdry}, $C_c$ is a face of $D_c$.
  \end{enumerate}
\end{remark}

\begin{defn}[Canonical retractions of bounded  linear cells]\label{def-cangen}
Let $\cal C$ be a special linear decomposition of some definable set,  $C, D\in \cal C$ with $C\sub cl(D)$, and $c$ a common corner of $C$ and $D$. Let $C'_c$ be the half-cell of $C_c$ and $C'=T^{-1}_{C, c}(C'_c)$. We call $C'$ the  \emph{$c$-half-cell} of $C$.

Now let $H_1:[0, q]\times cl(D_c)\to cl(D_c)$ be the canonical retraction of $cl(D_c)$ to $cl(C'_c)$. We define the \emph{$c$-canonical retraction of $cl(D)$ to $C'$} to be the map $H_c:[0, q]\times cl(D)\to cl(D)$, given by:
$$H_c(t, -)=T_{D, c}^{-1} \circ H_1(t,-) \circ T_{D, c}.$$
\end{defn}

\begin{claim}\label{contract2} Let $c, C$ be as above. The $c$-half-cell of $C$ is definably contractible.
\end{claim}
\begin{proof}
  By Lemma \ref{contract}.
\end{proof}

\begin{claim}\label{Hinside2}
Let $C, D$, $c$, $C'$ and $H_c$ be as above. Then $(H_c)_{\res [0, q]\times (C\cup D)}$ is a deformation retraction of $C\cup D$ to $C'$.
\end{claim}
\begin{proof}
 By Lemma \ref{Hinside} and Remark \ref{canret2}(2).
\end{proof}

\begin{claim}\label{extend2} Let $\cal C$ be a special linear decomposition of a definable set, and $C, E, D\sub \cal C$ three bounded cells. Assume that $C\sub cl(E)\sub cl(D)$, and that $c$ a common corner of $C, E, D$. Let $H_c$ and $H_c'$ be the $c$-canonical retractions of $cl(E)$ and $cl(D)$ to the $c$-half-cell $C'$ of $C$, respectively. Then $H_c'$ extends $H_c$.
\end{claim}
\begin{proof}
  By Lemma \ref{extend} and Remark \ref{canret2}(2).
\end{proof}

 \section{The proofs of Propositions 1 and 2}\label{sec-thms}

 We begin with Proposition 1.

\begin{prop}\label{main1}
Let $Y\sub R^n$ be a bounded definable set, $C\sub Y$ a linear cell, and $\cal D$ a special linear decomposition of $Y$ that contains $C$.  Assume that
$$\forall D\in \cal D,\,\, C\cap cl(D)\ne \emptyset.$$
Let $c$ be a corner of $C$, and $C'$ the $c$-half-cell of $C$. Then $Y$ deformation retracts to $C'$. In particular,  $Y$ is definably contractible.
\end{prop}

\begin{proof} Let $\cal D=\{D_1, \dots, D_k\}$.
By Lemma \ref{specialcorner}, $c$ is also a corner of each $D_i$. For every $i$, let
$$H^c_i:[0, q_i]\times cl(D_i)\to cl(D_i)$$
be the $c$-canonical retraction of $cl(D_i)$ to $cl(C')$.  Let $q=\max_i q_i$ and define
$$H:[0, q]\times  Y\to  Y$$
via
 $$H(t, x)=H^c_i(\min\{t, q_i\}, x),\, \text{ where $x\in D_i$}.$$
By Lemma \ref{Hinside2}, $H$ is indeed a map with image in $Y$.  Moreover, it is clear that  $H$ is  a definable  map that satisfies properties (1) - (3) from Definition \ref{def-retraction}. So we only need to prove that it is continuous. For that, we need to check that if $D_i, D_j\in \cal D$ with $D_i\cap cl(D_j)\ne \emptyset$, then $H^c_i$ and $H^c_j$ agree on  $D_i\cap cl(D_j)$. But by Fact \ref{fact-special2}, $D_i\sub cl(D_j)$,
 and hence, by Claim \ref{extend2}, $H^c_j$ extends $H^c_i$.

The last clause is by Claim \ref{contract2}.
\end{proof}

For the proof of Proposition 2 we will need the following lemma.

\begin{lemma}\label{formain2}
Let $Y\sub R^n$ be a definable set and $\cal C$  a special linear decomposition of $Y$. Then for every box $B_1\sub R^n$, there is a bigger box $B\supseteq B_1$ such that the collection
$$\cal D=\{B\cap D:\, D\in \cal C\}$$
is a special linear decomposition of $B\cap Y$.
\end{lemma}
\begin{proof}
By induction on $n$. For $n=1$, let  $B=B_1=(a, b)$, where $a, b\in R$. Then $\cal D$ is a linear  decomposition of $B\cap Y$ and hence is special.

Let $n>1$, and assume that $B_1=\pi(B_1)\times (d, e)$. By induction, there is a box $A\sub R^{n-1}$ containing $\pi(B_1)$ such that the collection
$$\cal E=\{A \cap E: E\in \pi(\cal C)\}$$
is a special linear decomposition of $A\cap \pi(Y)$. Let $\cal F$ be the set of all linear maps that appear in the definitions of cells in $\cal C$, with images in $R$. Now take $d', e'\in R$ such that
$$d' \text{ is smaller than $d$ and all } Im f_{| A},\, \text{ for } f\in \cal F,$$
$$e' \text{ is bigger than $e$ and all }  Im f_{| A},\, \text{ for } f\in \cal F,$$
which exist since $A$ is bounded and each $f\in \cal F$ is linear. Let $B=A\times (d, e)$. So $\cal D = \{B\cap D : D \in \cal C\}$ consists of linear cells with domain in $\cal E$ and cell-maps  already in $\cal F$, plus some linear cells of the form
$$(*)\,\,\,\,\, (d', g)_V \text{ or } (f, e')_V,$$
where $f, g\in \cal F$ and $V\in \cal E$.  We prove that \cal D is a special linear decomposition of $B\cap Y$. Since $\cal E$ is special, we only need to check that:\smallskip

\noindent (**) For every two cells $D_1=\Gamma(f)_S$ and $D_2=\Gamma(g)_T$ in $\cal D$, and $V\in \pi(\cal D)=\cal E$,
$$f_{|V}<g_{|V} \text{ or } f_{|V}=g_{|V} \text{ or } f_{|V}>g_{|V}.$$

\noindent (***) For every two cells $D_1=\Gamma (h)_S, \, D_2=(f,g)_T \in \cal D$,
$$ \text{ there is no $c\in cl(S)\cap cl(T)$ such that $f(c)<h(c)<g(c)$.}$$

For (**): since $B$ is open, $f, g$ must belong to $\cal F$, whereas $V=A\cap V'$ for some $V'\in \pi(\cal C)$. Hence the result is immediate from Definition \ref{def-special}(2) for $f, g\in \cal F$ and $V'$.

For (***): again, since $B$ is open, $h\in \cal F$. So, if $f, g\in \cal F$, we already know it. If not, then by (*), either $f=d'$ or $g=e'$, say the former. So $D_2=(d', g)_T = (-\infty, g)_T\cap B$, where $(-\infty, g)_T\sub (-\infty, g)_{T'}\in \cal C$, for some $T'\in \pi(\cal C)$. Also, $\Gamma(h)_S\sub \Gamma(h)_{S'}$, for some $S'\in \pi(\cal C)$. By the choice of $d'$, $f=d'<h_{|S}$, whereas applying Definition \ref{def-special}(3) for the cells  $\Gamma(h)_{S'}$ and $(-\infty, g)_{T'}$ of \cal C, we obtain (***).
\end{proof}



We finally derive Proposition 2.

\begin{prop}\label{main2}
Let $Y\sub R^n$ be a definable set, $C\sub Y$ a linear cell, and $\cal D$ a special linear decomposition of $Y$ that contains $C$. Assume that
$$\forall D\in \cal D,\,\, C\cap cl(D)\ne \emptyset.$$
Then $Y$ is  definably simply-connected.
\end{prop}
\begin{proof}
Consider a definable loop $\gamma:[0, \alpha]\to Y$.  By continuity of $\gamma$,
$Im(\gamma)$ is bounded, and hence it is contained in some box $B_1\sub R^n$ with $B\cap C\ne\emptyset$. By Lemma \ref{formain2}, there is a bigger box $B\supseteq B_1$ such that
$$\cal D' =\{B\cap D:\, D\in \cal D\}$$
is a special linear  decomposition of $B\cap Y$. We now observe that $Im(\gamma)\sub \cup \cal D'\sub Y$, and
$$\forall D\in \cal D',\,\, C\cap cl(D)\ne \emptyset.$$
Therefore, by Proposition \ref{main1}, $\cup \cal D'$ is definably contractible. Hence, by Claim \ref{contractible1}, $\gamma$ is definably homotopic in $\cup \cal D'$ to a constant loop. Therefore it  is definably homotopic in $Y$ to a constant loop.

It is also clear that $Y$ is definably connected, since $\cal D$ contains the linear cell $C$ and for every $D\in \cal D$, $C\cap cl(D)\ne \emptyset$.
\end{proof}

\section{Appendix}

\noindent\textbf{Fact \ref{fact-special1}.} \emph{Let $Y\sub R^n$ be a definable set. Then for any linear decomposition $\mathcal{D}$ of $Y$, there is a special linear decomposition $\mathcal{C}$ of $Y$ that refines $\mathcal{D}$ (that is, every linear cell in $\mathcal{D}$ is a union of linear cells in $\mathcal{C}$).}

\begin{proof}
The proof is almost word-by-word that of \cite[Lemma 2.6]{eep2}, which covered the case $Y=R^n$. By induction on $n$. For $n=1$, take $\mathcal{C}=\mathcal{D}$. Now assume that $n>1$ and the lemma holds for $n-1$. Let $\mathcal{D}$ be a linear decomposition of $Y$. Choose a finite collection $\mathcal{F}$ of linear maps $f:R^{n-1}\to R$ such that any linear map that appears in the definition of a linear cell from $\mathcal{D}$ is a restriction of some map from  $\mathcal{F}$.
Now set
\[
\mathcal{G}=\{\Gamma(f)\cap\Gamma(g): f, g\in \mathcal{F}\} \text{ and } \mathcal{G}^\prime = \{\pi (A): A\in \mathcal{G}\} \cup \pi ({\mathcal D}).
\]
Clearly, $\mathcal{G}^\prime$ is a finite collection of definable subsets of $R^{n-1}$. By the linear cell decomposition theorem and inductive hypothesis, there is a special
linear decomposition $\mathcal{C}^\prime$ of $R^{n-1}$ that partitions each member of $\mathcal{G}^\prime$.\vskip.2cm

\noindent\textbf{Claim.}
For any $f,g\in \mathcal{F}$, either $f< g$ or $f= g$ or $f> g$ on any $V\in \mathcal{C}'$.

\pf[Proof of Claim]
Let $V\in \mathcal{C}'$ and let $A=\Gamma(f)\cap \Gamma(g)$. Since $\pi(A)$ is a union of members of $\mathcal{C}'$, we have either $V\subseteq \pi (A)$ or $V\cap \pi (A)=\emptyset .$ In the first case $f=g$ on $V.$ In the second case, $V$ is a disjoint union of the open definable subsets $\{b\in V: f(b)<g(b)\}$ and $\{b\in V: g(b)<f(b)\}.$ Since $V$ is definably connected, one of the two sets is equal to $V$.
\qed
\vskip.2cm

Let $\mathcal{C}$ be the linear decomposition of $Y$ with $\pi (\mathcal{C})=\mathcal{C}'$ such that for any $V\in \mathcal{C}'$ the set of cells in $\mathcal{C}$ with domain $V$ is defined by all functions from $\mathcal{F}.$ Since $\mathcal{C}'$ refines $\pi (\mathcal{D})$, the choice of $\mathcal{F}$  and  Claim imply that $\mathcal{C}$ refines $\mathcal{D}.$

To conclude, we need to prove (2) and (3) from Definition \ref{def-special}. Item (2) simply holds by the Claim. For (3), let $(f,g)_T\in \mathcal{C}.$ Then $f,g \in \mathcal{F}$ and for any $h\in \mathcal{F}$, again from Claim, we have on $T$ either $h< f$, or $h=f$, or $h= g$ or $h> g,$ and so either $h(c)\leq f(c)$ or $g(c)\leq h(c)$, for any $c\in cl(T).$ In particular,  for any $\Gamma (h)_S\in \mathcal{C}$ there is no $c\in cl(S)\cap cl(T)$ such that $f(c)<h(c)<g(c).$
\end{proof}


\end{document}